\documentclass[11pt,reqno]{amsart}
\usepackage{amsmath,amssymb,amsthm,upref,graphicx,mathrsfs}

\usepackage{color,xcolor}
\usepackage[
  colorlinks=true,
  linkcolor=blue,
  citecolor=blue,
  urlcolor=blue]{hyperref}

\newtheorem{thm}{Theorem}[section]
\newtheorem{cor}[thm]{Corollary}
\newtheorem{lem}[thm]{Lemma}
\newtheorem{prop}[thm]{Proposition}

\numberwithin{equation}{section}

\begin{document}

\leftline{ \scriptsize}

\vspace{1.3 cm}
\title
{Some remarks on permanent dominant conjecture}
\author{ Kijti Rodtes $^{\ast}$ }
\thanks{{\scriptsize
		E-mail addresses:  kijtir@nu.ac.th (Kijti Rodtes).\\
		$^{*}$Research center for Academic Excellence in Mathematics, Department of Mathematics, Faculty of Science, Naresuan University, Phitsanulok 65000, Thailand.\\
		\\}}
\hskip -0.4 true cm

\maketitle


\begin{abstract} In this paper we provide an identity between determinant and generalized matrix function.  Also, a criterion of positive semi-definite matrices affirming the permanent dominant conjecture is given.  As a consequence, infinitely many infinite classes of positive semi-definite serving the conjecture (does not depend on groups or characters) are provided by generating from any positive semi-definite matrix having no zero in the first column.
\end{abstract}

\vskip 0.2 true cm


\pagestyle{myheadings}
\markboth{\rightline {\scriptsize Kijti Rodtes}}
{\leftline{\scriptsize }}
\bigskip
\bigskip


\vskip 0.4 true cm

\section{Introduction}

Throughout this paper, we denote $M_n(\mathbb{C})$ the set of square matrices over the field of complex numbers $\mathbb{C}$.  If $A\in M_n(\mathbb{C})$ is a positive semi-definite matrix, then it is well known that (see, for example, Chapter 7 in \cite{MerrisB}) $h(A)\geq\det(A)\geq 0$ (Hadamard inequality) and $\det(B)\det(C)\geq \det(A)$ (Fischer inequality), where $h(A)=\prod_{i=1}^n A_{ii}$ is the product of the main diagonal entries of $A=(A_{ij})$ and $B, C$ are square main diagonal blocks (any size) of $A$.  For a finite group $G$ embedded in the full symmetric group $S_n$ and character $\chi$ of $G$, Issai Schur (\cite{Schur}, 1918), introduced the notion of \textit{generalized matrix function }$d^G_\chi:M_n(\mathbb{C})\longrightarrow \mathbb{C}$ as 
$$d^G_\chi(A)=\sum_{g\in G}\chi(g)\prod_{i=1}^n A_{ig(i)}.  $$
If $G=S_n$ and $\pi$ is a character of $S_n$, then $d^G_\chi(A)$ is called an \textit{immanant}.  Permanent and determinant are immanants with $\chi=1$ (the principle character) and $\chi=\epsilon$ (alternating character), respectively.  Also, $h(A)=d^{\{e\}}_1(A)$ and $\det(B)\det(C)=d^{S_m\times S_k}_{(\epsilon_1,\epsilon_2)}(A)$, where $\{e\}$ is the trivial group and $m,k$ are sizes of square diagonal  blocks $B,C$ of $A$ and $\epsilon_1,\epsilon_2$ are alternating characters of $S_m,S_k$, respectively.  Issai Schur also extended the Hadamard  inequality and Fisher inequality as the \textit{Schur  inequality}, \cite{Schur}: $$\bar{d}^G_\chi(A)\geq \det(A),  $$ for any positive semi-definite matrix $A\in M_n(\mathbb{C})$ and character $\chi$ of $G$, where $\bar{d}^G_\chi(A)$ is the normalized of $d^G_\chi(A)$, i.e., $\bar{d}^G_\chi(A):=(1/\chi(e))d^G_\chi(A)$.

In the opposite direction of the Schur inequality, Marvin Marcus showed in 1963 that, \cite{Marcus}, $$\operatorname{per}(A)\geq h(A),$$ for any positive semi-definite matrix $A$. Three years later,  Elliott Hershel Lieb obtained that, \cite{Lieb}, $$\operatorname{per}(A)\geq \operatorname{per}(B) \operatorname{per}(C), $$
where $B,C$ are square diagonal blocks of the positive semi-definite matrix $A$.  He also raised the \textit{Permanent Dominant Conjecture} as: 
$$ \bar{d}^G_\chi(A)\leq \operatorname{per}(A), $$
for any positive semi-definite matrix $A\in M_n(\mathbb{C})$ and irreducible character $\chi$ of $G\leq S_n$.

This conjecture still opens up to this date.  There are progresses on this conjecture but most of them concentrate on  particular groups and characters; especially on immanants, which is known to satisfy the conjecture when $n\leq 13$, \cite{Pate}.  Also, according to Lieb inequality, the conjecture holds true for any Young subgroup of $S_n$ and $\chi=1$, \cite{merrisandwatkin}.  Further details of the progress to this conjecture along this direction and related conjectures can be found, for example, in \cite{Wanless}, \cite{Zhang} and references therein. 

On the other hand, it is quite obvious to see that, for any positive semi-definite matrix $A\in M_n(\mathbb{C})$, if $A$ is a non-negative real  matrix, or $A$ has rank $1$, or $A$ contains a zero row (or column), or $A$ has a zero submatrix of dimension $r\times s$ with $r+s\geq n+1$, then $A$ satisfies the conjecture (see more details in the next section). However, besides these classes of matrices, it seems that there is no new explicit class of positive semi-definite matrices provided to serve the conjecture.  In this paper, we use generalized Cauchy-Binet theorem together with character theory to see some relationship (identity) between determinant and generalized matrix functions, in Theorem \ref{maintheorem}.  A criterion of positive semi-definite matrices affirming the conjecture is given in Corollary \ref{cor2}.  As a consequence, infinitely many infinite classes of positive semi-definite serving the conjecture are generated from any positive semi-definite matrix having no zero in the first column, in Theorem \ref{main3}.

\section{Some basic remarks}
Throughout this section, we let $A=(A_{ij})\in M_n(\mathbb{C})$ be a positive semi-definite matrix and let $G\leq S_n$ be a subgroup of $S_n$.  For a character $\chi$ of $G$, it is well known in character theory that $\chi$ is a linear combination of irreducible characters of $G$ with non-negative integer coefficients.  We then have that:
\begin{prop}\label{pp1}
	If $\bar{d}^G_\chi(A)\leq \operatorname{per}(A)$ for all irreducible characters $\chi$ of $G$, then $\bar{d}^G_\lambda(A)\leq \operatorname{per}(A)$ for all characters $\lambda$ of $G$.
\end{prop}

For each $g\in G$, denote $l_g(A):=\prod_{i=1}^n A_{ig(i)}$; so $d^G_\chi(A)=\sum_{g\in G}\chi(g)l_g(A)$.  It is also well known in character theory that $|\chi(g)|\leq \chi(e)$ for all $g\in G$, where $e$ is the identity of $G$.  Then $|\chi(g)|/\chi(e)\leq 1$ which yields that
$$ \bar{d}^G_\chi(A)=\frac{1}{\chi(e)}\sum_{g\in G}\chi(g)l_g(A)\leq \sum_{g\in G}\frac{|\chi(g)|}{\chi(e)}|l_g(A)|\leq \sum_{g\in G} |l_g(A)|\leq \sum_{\sigma\in S_n} |l_\sigma(A)|.  $$
We then have that:
\begin{prop}\label{posdia}
If $l_\sigma(A)\geq 0$ for all $\sigma\in S_n$, then $\bar{d}^G_\chi(A)\leq \operatorname{per}(A)$ for any character $\chi$ of $G$.  In particular, if $A$ is a non-negative real matrix (including diagonal positive semi-definite matrix), then $A$ satisfies the permanent dominant conjecture.
\end{prop}

We observe that if $A$ contains a zero row or zero column,  then $l_\sigma(A)=0$ for all $\sigma\in S_n$ (because $\sigma:[n]\longrightarrow [n]$ is one to one and onto function, where $[n]:=\{1,2,\dots,n\}$).  Also, by Frobenius-Konig theorem (see, for example, Theorem 5.20 in \cite{Zhang2}), $l_\sigma(A)=0$ for all $\sigma\in S_n$ if and only if $A$ contains an $r\times s$ zero submatrix, where $r+s=n+1$.  We then have that:
\begin{prop}
	If $A$ has a zero row or zero column or $A$ contains an $r\times s$ zero submatrix, where $r+s=n+1$, then $A$ satisfies the permanent dominant conjecture.
\end{prop}

Let $O_n$ be the set of all odd permutations in $S_n$.  By Schur inequality, $\det(A)\leq \operatorname{per}(A)$; namely, $$\sum_{\sigma\in A_n}l_\sigma(A)-\sum_{\sigma\in O_n}l_\sigma(A)\leq \sum_{\sigma\in A_n}l_\sigma(A)+\sum_{\sigma\in O_n}l_\sigma(A),    $$
which means that $\sum_{\sigma\in O_n}l_\sigma(A) \geq 0$.  We can conclude that:
\begin{prop}\label{propan}
	For the alternating group $A_n$ and the principle character $1$, 
$$d^{A_n}_1(A)\leq \operatorname{per}(A).$$ 
\end{prop}

Furthermore, if $\operatorname{rank}(A)=1$, then $A=vv^*$ for some $v=(a_1,a_2,\dots,a_n)^t\in \mathbb{C}^n$ (since $A$ is a positive semi-definite matrix). Then, for each $\sigma\in S_n$, $$l_\sigma(A)=\prod_{i=1}^nA_{i\sigma(i)}=\prod_{i=1}^na_i\bar{a}_{\sigma(i)}=(\prod_{i=1}^n a_i)(\prod_{i=1}^n \bar{a}_{\sigma(i)})=(\prod_{i=1}^n a_i)\overline{(\prod_{i=1}^n a_i)}=|\prod_{i=1}^n a_i|^2\geq 0. $$
By Proposition \ref{posdia}, we can conclude that:

\begin{prop}
If positive semi-definite matrix $A\in M_n(\mathbb{C})$ has $\operatorname{rank}(A)=1$, then $A$ satisfies the permanent dominant conjecture. 
\end{prop}
Moreover, the above conclusion can be extended to any sum of rank one matrix and non-negative diagonal real matrix as the following proposition. 

\begin{prop}
	If  $A=vv^*+\operatorname{diag}(d_1,d_2,\dots,d_n)$, where $v\in \mathbb{C}^n$ and $d_i\geq 0$ for all $i$, then $A$ satisfies the permanent dominant conjecture. 
\end{prop} 
\begin{proof}
	Let $v=(a_1,a_2,\dots,a_n)^t$ and $\sigma\in S_n$. Then,
	$$l_\sigma(A)=\prod_{i=1}^n (a_i\bar{a}_{\sigma(i)}+\delta_{i\sigma(i)}d_i)=\sum_{I\subseteq [n]}(\prod_{i\in I}a_i\bar{a}_{\sigma(i)})(\prod_{i\notin I}\delta_{i\sigma(i)}d_i).  $$
	The non-zero summand exists when the index set $I$ satisfies $\sigma(i)=i$ for all $i\notin I$; denote the family of all such index sets by $N^c$.  Hence, $\sigma(I)=I$ for all $I\in N^C$ and thus
	$$l_\sigma(A)= \sum_{I\in N^c}(\prod_{i\in I}a_i\bar{a}_{\sigma(i)})(\prod_{i\notin I}d_i)=\sum_{I\in N^c}(\prod_{i\in I}a_i)(\overline{\prod_{i\in I}a_i})(\prod_{i\notin I}d_i)=\sum_{I\in N^c}|\prod_{i\in I}a_i)|^2(\prod_{i\notin I}d_i)\geq 0. $$
	By Proposition \ref{posdia}, the proof is completed.
\end{proof}

\section{Classes of matrices affirming the conjecture}
Throughout this section, let $G$ be a finite group embedded in the full symmetric group $S_n$ and let $\chi$ be an irreducible character (need not be linear) of $G$.  We first recall the Cauchy-Binet Theorem for generalized matrix function: 
\begin{thm}[\cite{MerrisB}, Theorem 7.34]\label{thm1}
  If $\alpha, \beta\in \Omega:=\{\gamma\in \Gamma_{m,n}\;|\;(\chi,1)_{G_\gamma}\neq 0 \}$, then
	$$ d^G_\chi((AB)[\alpha|\beta]) =\frac{\chi(e)}{|G|}\sum_{\gamma\in \Omega}d^G_\chi(A[\alpha|\gamma])d^G_\chi(B[\gamma|\beta]),$$
	for any $A,B\in M_n(\mathbb{C})$, where $\Gamma_{m,n}=\{\gamma:= (\gamma_1,\gamma_2,\dots,\gamma_m)\in \mathbb{N}^m\;|\; 1\leq \gamma_i\leq n\}$, $G_\gamma=\{g\in G\;|\; g\cdot\gamma:=(\gamma_{g(1)},\gamma_{g(2)},\dots,\gamma_{g(m)})=\gamma \}$, $(\chi,1)_{G_\gamma}=\frac{1}{|G_\gamma|}\sum_{g\in G_\gamma}\chi(g)$ and $AB[\alpha|\beta]$ is the matrix obtained from $AB$ in which  the rows and columns are indexed by $\alpha$ and $\beta$, respectively.
\end{thm}
Let $A$ be a positive semi-definite matrix in $M_n(\mathbb{C})$.  Then, by the Cholesky decomposition, there exits a lower triangular matrix $L$ such that $A=LL^*$.  We observe that, for $\gamma, (n):=(1,2,\dots,n) \in \Gamma_{n,n}$,
\begin{eqnarray*}
d^G_\chi(L^*[\gamma|(n)])&=&\sum_{\sigma\in G}\chi(\sigma)\prod_{j=1}^n(L^*)_{\gamma_j\sigma(j)} \\
&=&\sum_{\sigma\in G}\chi(\sigma)\overline{\prod_{j=1}^n(L^t)_{\gamma_j\sigma(j)}} \\
&=&\overline{\sum_{\sigma\in G}\overline{\chi(\sigma)}\prod_{j=1}^n(L^t)_{\gamma_j\sigma(j)}} \\
&=&\overline{\sum_{\sigma\in G}\chi(\sigma^{-1})\prod_{j=1}^nL_{\sigma(j)\gamma_j} }\\
&=&\overline{\sum_{\sigma\in G}\chi(\sigma^{-1})\prod_{j=1}^nL_{i\gamma_{\sigma^{-1}(i)}} }\\
&=&\overline{\sum_{\theta\in G}\chi(\theta)\prod_{j=1}^nL_{i\gamma_{\theta(i))}}}=\overline{d^G_\chi(L[(n)|\gamma])}.
\end{eqnarray*}
By Theorem \ref{thm1} with $G\leq S_n$ and $\alpha=\beta=(n)$, we now compute that
$$ d^G_\chi(A) =\frac{\chi(e)}{|G|}\sum_{\gamma\in \Omega}d^G_\chi(L[(n)|\gamma])d^G_\chi(L^*[\gamma|(n)])=\frac{\chi(e)}{|G|}\sum_{\gamma\in \Omega}|d^G_\chi(L[(n)|\gamma])|^2;$$
namely,
\begin{equation}\label{eq1}
\bar{d}^G_\chi(A)=\frac{1}{|G|}\sum_{\gamma\in \Omega}|d^G_\chi(L[(n)|\gamma])|^2.
\end{equation}

Note also that, for any $g\in S_n$, $L[(n)|g\cdot \gamma]=L[(n)|\gamma]P_g$, where $P_g=(\delta_{ig(j)})$ is the permutation matrix associated to $g\in S_n$.  Then, by setting $L^\omega:=L[(n)|\omega]$ (for $\omega\in \Gamma_{n,n}$), we have that $$d^G_\chi(L^{g\cdot \gamma})=d^G_\chi(L^\gamma P_g)$$ for any $g\in G$, where $g\cdot \gamma:=(\gamma_{g(1)}, \gamma_{g(2)},\dots,\gamma_{g(n)})$ (the standard action of $G$ on $\Gamma_{n,n}$).  Let $\Delta_G$ be the set of all representatives (first member of each orbit ordered by lexicographic ordering) of the action of $G$ on $\Gamma_{n,n}$ and let $\bar{\Delta}_G=\Delta_G\cap \Omega$.  For each $\gamma\in \bar{\Delta}_G$  denote the set of all representatives of  left cosets of $G_\gamma$ in $G$ by   $S_\gamma$. Then, for each orbital $[\gamma]$,
\begin{equation} \label{eq21}
\sum_{\omega\in [\gamma] }|d^G_\chi(L^{\omega})|^2=\sum_{g\in S_\gamma}|d^G_\chi(L^{\gamma} P_g)|^2.
\end{equation}
   According to $\Omega=\cup_{\gamma\in \bar{\Delta}_G}[\gamma]$ (Lemma 6.22 in \cite{MerrisB}), the relation (\ref{eq1}) becomes
\begin{equation}\label{eq2}
\bar{d}^G_\chi(A)=\sum_{\gamma\in \bar{\Delta}_G}\frac{1}{|G|}\sum_{g\in S_\gamma}|d^G_\chi(L^{\gamma} P_g)|^2.
\end{equation}

Now, we consider the sequence $\gamma^\circ:=(n)$ which clearly belongs to $\bar{\Delta}_G$ and $S_{\gamma^\circ}=G$.  Since $L$ is a lower triangular matrix, 
\begin{eqnarray*}
\frac{1}{|G|}\sum_{g\in S_{\gamma^\circ}}|d^G_\chi(L^{\gamma^\circ} P_g)|^2 
&=&\frac{1}{|G|}\sum_{g\in G}|d^G_\chi(L P_g)|^2 \\ 
&=&\frac{1}{|G|}\sum_{g\in G}| \sum_{h\in G}\chi(h)\prod_{i=1}^{n}(LP_g)_{ih(i)}|^2 \\ 
&=&\frac{1}{|G|}\sum_{g\in G}| \sum_{h\in G}\chi(h)\prod_{i=1}^{n}L_{igh(i)}|^2 \\ 
&=&\frac{1}{|G|}\sum_{g\in G}|\sum_{\theta\in G}\chi(g^{-1}\theta)\prod_{i=1}^{n}L_{i\theta(i)}|^2 \\
&=&\frac{1}{|G|}\sum_{g\in G}|\chi(g^{-1})\prod_{i=1}^{n}L_{ii}|^2 \\
&=&\frac{1}{|G|}\sum_{g\in G}|\chi(g^{-1})|^2|\prod_{i=1}^{n}L_{ii}|^2 \\
&=&\frac{1}{|G|}\sum_{g\in G}|\chi(g^{-1})|^2(\det(L)\det(L^*)) \\
&=&(\frac{1}{|G|}\sum_{g\in G}|\chi(g)|^2)\det(A).
\end{eqnarray*}
Since $\chi$ is an irreducible character, $(\chi,\chi)_G=1$ which means that $$1=(\chi,\chi)_G:=\frac{1}{|G|}\sum_{g\in G}\overline{\chi(g)}\chi(g)=\frac{1}{|G|}\sum_{g\in G}|\chi(g)|^2. $$ So, 
$$ \frac{1}{|G|}\sum_{g\in S_{\gamma^\circ}}|d^G_\chi(L^{\gamma^\circ} P_g)|^2 =\frac{1}{|G|}\sum_{g\in G}|d^G_\chi(L P_g)|^2 =\det(A).$$
Moreover, $\sigma\cdot \gamma^\circ$ for $\sigma\in S_G:=\{\sigma_1,\dots,\sigma_t\}$, the set of all representatives for the right cosets of $S_n$ by $G$,  are in different orbits and all entries of the sequence $\sigma\cdot \gamma^\circ$ are distinct.  If $\sigma\in S_G\setminus G$, then
\begin{eqnarray*}
	\frac{1}{|G|}\sum_{g\in S_{\sigma\cdot\gamma^\circ}}|d^G_\chi(L^{\sigma\cdot\gamma^\circ} P_g)|^2 
	&=&\frac{1}{|G|}\sum_{g\in S_{\sigma\cdot\gamma^\circ}}|d^G_\chi(L P_{\sigma g})|^2 \\ 
	&=&\frac{1}{|G|}\sum_{g\in S_{\sigma\cdot\gamma^\circ}}|  \sum_{h\in G}\chi(h)\prod_{i=1}^{n}(L P_{\sigma g})_{ih(i)}|^2\\
	&=&\frac{1}{|G|}\sum_{g\in S_{\sigma\cdot\gamma^\circ}}|\sum_{h\in G}\chi(h)\prod_{i=1}^{n}L_{i\sigma gh(i)}|^2 \\
	&=&0,
\end{eqnarray*}
because $\sigma g h\neq e$ for any $\sigma\notin G$ and $g,h\in G$.

Next, we consider the sequence $\gamma^k:=(k,k,\dots,k)$, for each $1\leq k\leq n$.  This sequence belongs to $\bar{\Delta}_G$ if and only if $\chi=1$.  Since $S_{\gamma^k}=\{e\}$, we get that
\begin{eqnarray*}
	\frac{1}{|G|}\sum_{g\in S_{\gamma^k}}|d^G_1(L^{\gamma^k} P_g)|^2 
	&=&\frac{1}{|G|}|d^G_1(L[(1,2,\dots,n),(k,k,\dots,k)])|^2\\ 
		&=&\frac{1}{|G|}|\sum_{g\in G} \prod_{i=1}^n L_{ig(i)}|^2 \\
	&=&\frac{1}{|G|}|\sum_{g\in G} \prod_{i=1}^n L_{ik}|^2 \\
	&=&\frac{1}{|G|}(|G| |\prod_{i=1}^n L_{ik}|)^2\\
	&=&\begin{cases}
		|G|\prod_{i=1}^n |L_{i1}|^2, & \text{if $k=1$ }\\
		0, & \text{otherwise.}
	\end{cases}
\end{eqnarray*}

Hence, by setting $$\bar{\Delta}^*_G=\bar{\Delta}_G-\{\sigma\cdot \gamma^\circ\;|\; \sigma\in S_G\}-\{\gamma^k\;|\; k=1,2,\dots,n\},$$
and using the above discussion together with relation (\ref{eq2}), we can conclude that:
\begin{thm}\label{maintheorem}
	Let $G$ be a finite subgroup of $S_n$ and $\chi$ be an irreducible character of $G$.  Let $A=LL^*$ be a positive semi-definite matrix in $M_n(\mathbb{C})$, where $L=(L_{ij})$ is a lower triangular matrix  from a Cholesky decomposition of $A$.  Then 
	$$  \bar{d}^G_\chi(A)=\det(A)+\delta_{\chi,1}	|G|\prod_{i=1}^n |L_{i1}|^2+\sum_{\gamma\in \bar{\Delta}^*_G}\frac{1}{|G|}\sum_{g\in S_\gamma}|d^G_\chi(L^{\gamma} P_g)|^2, $$
	where $\delta_{\chi,1}$ is the Kronecker delta function.
\end{thm}
This result immediately implies that:
\begin{cor} [Schur inequality]
	Let $\chi$ be an irreducible character of a finite subgroup $G$ of $S_n$.  Then     $$ \bar{d}^G_\chi(A)\geq \det(A), $$
	for any  positive semi-definite matrix $A\in M_n(\mathbb{C})$.
\end{cor}
In particular, when $\chi$ is linear, it turns out that $$	d^G_\chi(L^\gamma P_g)=\chi(g^{-1})\sum_{\theta\in G}\chi(\theta)\prod_{i=1}^{n}L^\gamma_{i\theta(i)}=\chi(g^{-1})d^G_\chi(L^\gamma).$$
Since $|\chi(g)|=1$ and $\chi(gh)=\chi(g)\chi(h)$ for all $g,h\in G$ (because $\chi$ is linear), $$  \sum_{g\in S_\gamma}|d^G_\chi(L^{\gamma} P_g)|^2 =\sum_{g\in S_\gamma}|\chi(g^{-1})d^G_\chi(L^\gamma)|^2=\frac{|G|}{|G_\gamma|}|d^G_\chi(L^\gamma)|^2 .$$ By using Theorem \ref{maintheorem} and using the same notations as above, we have that: 
\begin{cor} 
	Let $\chi$ be a linear character of a finite subgroup $G$ of $S_n$.  Then     $$ \bar{d}^G_\chi(A)=\det(A)+\delta_{\chi,1}	|G|\prod_{i=1}^n |L_{i1}|^2+\sum_{\gamma\in \bar{\Delta}^*_G}\frac{1}{|G_\gamma|}|d^G_\chi(L^{\gamma} )|^2.$$
	In particular, 
	$$\operatorname{per}(A)=\det(A)+n!\prod_{i=1}^n |L_{i1}|^2+(n-1)!\prod_{i=1}^{n-1} |L_{i1}|^2|L_{nn}|^2+\sum_{\gamma\in \bar{\Delta}^{**}_{S_n}}\frac{1}{|(S_n)_\gamma|}|\operatorname{per}(L^{\gamma} )|^2,$$
	where $\bar{\Delta}^{**}_{S_n}=\bar{\Delta}^*_{S_n}-\{(1,1,\dots,1,n)\}$.
\end{cor}
\begin{proof}
	The term $(n-1)!\prod_{i=1}^{n-1} |L_{i1}|^2|L_{nn}|^2$ comes from the evaluating $\frac{1}{|(S_n)_\gamma|}|\operatorname{per}(L^{\gamma} )|^2$ with $\gamma=(1,1,\dots,1,n)$ directly.
\end{proof}

In order to bound the value of $\bar{d}^G_\chi(A)$, by Theorem \ref{maintheorem}, it is reasonable to concentrate only on $\gamma\in \bar{\Delta}^*_G$ for which $d^G_\chi(L^\gamma P_g)\neq 0$ for some $g\in S_\gamma$; denote the set of all such sequences by $\hat{\Delta}^*_G$.  We have that:
\begin{lem}\label{lem1}
	By using the same notions as above and $n> 3$, $$|\hat{\Delta}^*_G|\leq M_n:= \begin{cases}
n^n-(n-1)^n-n!-n^{\frac{n-4}{2}}(\frac{n-2}{2})^{{\frac{n+2}{2}}}(\frac{n+4}{2})-1, & \text{if $n$ is even }\\
		n^n-(n-1)^n-n!-n^{\frac{n-5}{2}}(\frac{n-1}{2})^{{\frac{n+1}{2}}}(\frac{n+3}{2})-1, & \text{if $n$ is odd.}
	\end{cases}$$
\end{lem}
\begin{proof}
	Let $A:=\{\sigma\cdot (1,2,\dots,n)\;|\; \sigma\in S_n\}$ and  $B:=\{\omega:= (\omega_1,\omega_2,\dots,\omega_n)\in \Gamma_{n,n}\;|\; 2\leq \omega_i\leq n, \forall i=1,2,\dots,n\}$.  If $n>3$ is even, we define $C^e$ to be the set of all sequences $\omega:= (\omega_1,\omega_2,\dots,\omega_n)\in \Gamma_{n,n}$ satisfying $ \omega_j=1$ for some  $ j\in [n]$  and there exist $J\subseteq [n]$ with $|J|=\frac{n}{2}+1$ such that $\omega_k\in \{\frac{n}{2}+1, \frac{n}{2}+2,\dots, n\}$ for all $k\in J$.  If $n>3$ is odd, we define $C^o$ to be the set of all sequences $\omega:= (\omega_1,\omega_2,\dots,\omega_n)\in \Gamma_{n,n}$ satisfying $ \omega_j=1$ for some  $ j\in [n]$  and there exist $J\subseteq [n]$ with $|J|=\frac{n+1}{2}$ such that $\omega_k\in \{\frac{n+1}{2}, \frac{n+1}{2}+1,\dots, n\}$ for all $k\in J$.   Since $L$ is a lower triangular matrix, for  any sequence $\gamma$ in $B$ or $C^e$ or  $C^o$, it turns out that $d^G_\chi(L^\gamma P_g)=0$ for all $g\in G$.  By setting $\Omega^*:=\Omega-A-B-C^{\delta(n)}$ ($C^{\delta(n)}=C^e$ if $n$ is even, and $C^{\delta(n)}=C^o$ if $n$ is odd), we have that $\hat{\Delta}^*_G\subseteq \Omega^*$.   By direct calculation, we have $|\Omega|\leq n^n$, $|A|=n!$, $|B|=(n-1)^n$, $|C^e|\geq n^{\frac{n-4}{2}}(\frac{n-2}{2})^{{\frac{n+2}{2}}}(\frac{n+4}{2}) $ and $|C^o|\geq n^{\frac{n-5}{2}}(\frac{n-1}{2})^{{\frac{n+1}{2}}}(\frac{n+3}{2}) $.  Due to the sets $A,B$, $C^{\delta(n)}$, $\{(1,1,\dots,1)\}$ are pairwise disjoint and $(1,1,\dots,1)\notin \hat{\Delta}^*_G$, the result follows.
\end{proof}
For each positive semi-definite matrix $A=LL^*$, denote $\alpha= \prod_{i=1}^n|L_{i1}|$ and $\alpha_0=\max_{1\leq i \leq n}|L_{i1}|$. For each finite group $G\leq S_n$, denote $$  \epsilon(n,G):=\frac{\alpha}{|G|}\sqrt{\frac{2n!}{3M_n}}.  $$   
Let $\bar{\epsilon}(n,G):=\min\{\frac{1}{\alpha_0^{n-1}} \epsilon(n,G),1\}$.  By using the above notations, the following theorem holds:
\begin{thm}\label{main2}
	If $A=LL^*$ with $|L_{kj}|\leq \bar{\epsilon}(n,G)$ for all $k,j$ with $j\neq 1$,  then $$\bar{d}^G_\chi(A)\leq \det(A)+(\delta_{\chi,1}|G|+\frac{2}{3}n!)\prod_{i=1}^n|L_{i1}|^2,  $$
	for any irreducible character $\chi$ of $G\leq S_n$ with $n> 3$.
\end{thm}
\begin{proof}	
	Let $\gamma=(\gamma_1,\dots,\gamma_n)\in \hat{\Delta}^*_G$.  Because of $(1,1,\dots,1)\notin \hat{\Delta}^*_G$, there exists $s$ such that $\gamma_s\neq 1$.  Since $|L_{kj}|\leq \epsilon(n,G)$ and $|L_{kj}|\leq 1$   for all $k,j$ such that $j\neq 1$, we compute that, for any $g\in G$, 
\begin{eqnarray*}
|d^G_\chi(L^\gamma P_g)|&=&|\sum_{h\in G} \chi(h)\prod_{i=1}^n L_{i\gamma_{gh(i)}}|\\
&\leq& \sum_{h\in G} |\chi(h)|\prod_{i=1}^n |L_{i\gamma_{gh(i)}}|\\
&\leq& \sum_{h\in G} |\chi(h)|\alpha_0^{n-1} |L_{i\gamma_s}|, \hbox{ for some $i$}\\
&\leq& \epsilon(n,G) \sum_{h\in G} |\chi(h)|.
\end{eqnarray*}
Then, by power mean inequality (or by Cauchy-Schwartz inequality),
$$|d^G_\chi(L^\gamma P_g)|^2\leq (\epsilon(n,G))^2 (\sum_{h\in G} |\chi(h)|)^2\leq (\epsilon(n,G))^2 (|G|\sum_{h\in G} |\chi(h)|^2) = (\epsilon(n,G))^2 |G|^2.   $$
Thus, for each $\gamma\in \hat{\Delta}^*_G $, $$\frac{1}{|G|} \sum_{g\in S_\gamma}|d^G_\chi(L^\gamma P_g)|^2 \leq \frac{1}{|G|}|S_\gamma|(\epsilon(n,G))^2 |G|^2\leq (\epsilon(n,G))^2 |G|^2,  $$
which also yields that, by Lemma \ref{lem1},
$$\sum_{\gamma\in \hat{\Delta}^*_G }\frac{1}{|G|} \sum_{g\in S_\gamma}|d^G_\chi(L^\gamma P_g)|^2\leq M_n (\frac{\alpha}{|G|}\sqrt{\frac{2n!}{3M_n}})^2 |G|^2=\frac{2}{3}n! \prod_{i=1}^n|L_{i1}|^2.  $$
By Theorem \ref{maintheorem}, the proof is completed.
\end{proof}

Now, by setting, for $n>3$, $$\epsilon_n:=\bar{\epsilon}(n,S_n)=\min\{\frac{\alpha}{\alpha_0^{n-1}}\sqrt{\frac{2}{3n! M_n}},1\},  $$
we see that $\epsilon_n\leq \bar{\epsilon}(n,G)$ for any finite group $G\leq S_n$.  Hence, the following holds:
\begin{cor}\label{cor2}
	Let $A=LL^*$ with $|L_{kj}|\leq \epsilon_n$ for all $k,j$ with $j\neq 1$.  Then $$\bar{d}^G_\chi(A)\leq \operatorname{per}(A),$$ for any group $G\leq S_n$ and any irreducible character $\chi$ of $G$.
\end{cor}
\begin{proof}
	Since $\bar{d}^G_\chi(A)\leq per(A)$ for any positive semi-definite matrix $A\in M_n(\mathbb{C})$ with $n\leq 3$, it suffices to consider the cases $n>3$. Thus, by Theorem \ref{main2}, we have
	\begin{eqnarray}\label{cons1}
	 \operatorname{per}(A)-\bar{d}^G_\chi(A)\geq (\frac{1}{3}n!-\delta_{\chi,1}|G|)\prod_{i=1}^n|L_{i1}|^2.   
	\end{eqnarray} When $|G|=n!$ and $\chi=1$,  $\bar{d}^G_\chi(A)$ becomes $\operatorname{per}(A)$ where the inequality is obviously true.  Also, when $|G|=n!$ but $\chi\neq 1$,  $\operatorname{per}(A)-\bar{d}^G_\chi(A)\geq (\frac{1}{3}n!)\prod_{i=1}^n|L_{i1}|^2\geq 0$.  Furthermore, when $|G|=\frac{n!}{2}$, it is well known that $G$ must be the alternating group $A_n$, which is proved that $d^{A_n}_1(A)\leq \operatorname{per}(A)$  in Proposition \ref{propan}.  If, however, $\chi\neq 1$, then $$\operatorname{per}(A)-\bar{d}^{A_n}_\chi(A)\geq (\frac{1}{3}n!)\prod_{i=1}^n|L_{i1}|^2\geq 0.$$  Hence, by Lagrange's Theorem ($|G|$ divides $|S_n|$), it remains only to consider groups $G$ with $|G|\leq \frac{n!}{3}$ and this is done by (\ref{cons1}).
\end{proof}
The criterion in Corollary \ref{cor2} provides us infinitely many infinite classes of positive semi-definite matrix affirming the conjecture.   
\begin{thm}\label{main3}
	Let $A=LL^*\in M_n(\mathbb{C})$ with $m:=\max\{|L_{ij}|\;|\; 2\leq i,j\leq n\}$ and $a:=\frac{\epsilon_n}{m}$.  If each entry in the first column of $A$ is not zero, then
	$$[A]:=\{L\operatorname{diag}(1,\lambda^2_2,\dots,\lambda^2_n)L^*\;|\; 0< \lambda_2,\dots,\lambda_n \leq a\}  $$
	is an infinite class of positive semi-definite matrix satisfying the permanent dominant conjecture.
\end{thm}
\begin{proof}
	Let $\lambda_2,\dots,\lambda_n$ be positive real numbers lying in $\{r\in \mathbb{R}\;|\; 0< r\leq a\}$ and let $D:=\operatorname{diag}(1,\lambda_2,\dots,\lambda_n)$.  Then, the first columns of the lower triangular matrices $L$ and $L_D:=LD$ are identical.  This yields that $\epsilon_n(L_D)=\epsilon_n(L):=\epsilon_n$.   Since each entry in the first column of $A$ is not zero, so is $L$.  Thus $\epsilon_n>0$; namely $[A]$ is infinite.  Also, for $2\leq i,j\leq n$, $$|(L_D)_{ij}|=\lambda_i|L_{ij}|\leq \frac{\epsilon_n}{m}|L_{ij}|\leq \epsilon_n.$$
	By Corollary \ref{cor2},  $L_D L_D^*=LD^2L^*=L\operatorname{diag}(1,\lambda^2_2,\dots,\lambda^2_n)L^*$ affirms the conjecture.
\end{proof}
Note that if $a\geq 1$, then such positive semi-definite $A$ satisfies the conjecture.  Furthermore, for any complex diagonal matrix $D:=\operatorname{diag}(c_1,\dots,c_n)$ and $\sigma\in G$,
$$l_\sigma(DAD^*)= \prod_{i=1}^n (c_i A_{i\sigma(i)}\overline{c_{\sigma(i)}})=|\prod_{i=1}^n c_i|^2\prod_{i=1}^nA_{i\sigma(i)}= |\prod_{i=1}^n c_i|^2l_\sigma (A),$$
which yields that $\bar{d}^G_\chi(DAD^*)=|\prod_{i=1}^n c_i|^2\bar{d}^G_\chi(A)$.  Hence, for $\prod_{i=1}^n c_i\neq 0$, we have that $\bar{d}^G_\chi(DAD^*)\leq \operatorname{per}(DAD^*)$ if and only if $\bar{d}^G_\chi(A)\leq \operatorname{per}(A)$.  In particular, for non-zero complex numbers $c_1,\dots,c_n$, any matrix in the form 
$$ \operatorname{diag}(c_1,\dots,c_n)X\operatorname{diag}(\bar{c}_1,\dots,\bar{c}_n), 
 $$
 where $X\in [A]$ (in Theorem \ref{main3}), satisfies the permanent dominant conjecture.

\section*{Acknowledgments}
The author would like to thank anonymous referee(s) for reviewing this manuscript.  He also would like to thank Naresuan University (NU), and National Science, Research and Innovation Fund (NSRF): Grant NO. FRB660001/0179, for financial support.
\subsection*{Data Availability}
Data sharing not applicable to this article as no datasets were generated or analyses during the current study.
\subsection*{Conflict of interest}
The author has no relevant financial or non-financial interests to disclose.

\end{document}